\documentclass[11pt]{amsart}
\usepackage{latexsym,amssymb,amsfonts,amsmath,graphicx,url,amsthm,fullpage}

\def\+{\includegraphics[scale=0.5]{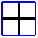}}
\def\bl{\includegraphics[scale=0.5]{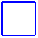}}
\def\bt{\includegraphics[scale=0.5]{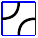}}
\def\rt{\includegraphics[scale=0.5]{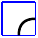}}
\def\jt{\includegraphics[scale=0.5]{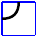}}
\def\htile{\includegraphics[scale=0.5]{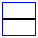}}
\def\vtile{\includegraphics[scale=0.5]{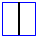}}

\usepackage{mathdots,verbatim}
\usepackage{tikz, calc, ifthen}
\usetikzlibrary{calc, shapes, backgrounds,arrows,positioning}
\tikzset{>=stealth',
  head/.style = {fill = white, text=black}, 
  pil/.style={->,thick},
  junct/.style = {draw,circle,inner sep=0.5pt,outer sep=0pt, fill=black}
  }
\definecolor{light-gray}{gray}{0.85}
\definecolor{dark-gray}{gray}{0.25}

\usepackage{tikz,tkz-graph}

\newcounter{x}
\newcounter{y}
\newcounter{z}

\newcommand\xaxis{210}
\newcommand\yaxis{-30}
\newcommand\zaxis{90}

\newcommand\topside[3]{
  \fill[fill=white, draw=black,shift={(\xaxis:#1)},shift={(\yaxis:#2)},
  shift={(\zaxis:#3)}] (0,0) -- (30:1) -- (0,1) --(150:1)--(0,0);
}

\newcommand\leftside[3]{
  \fill[fill=black, draw=black,shift={(\xaxis:#1)},shift={(\yaxis:#2)},
  shift={(\zaxis:#3)}] (0,0) -- (0,-1) -- (210:1) --(150:1)--(0,0);
}

\newcommand\rightside[3]{
  \fill[fill=gray, draw=black,shift={(\xaxis:#1)},shift={(\yaxis:#2)},
  shift={(\zaxis:#3)}] (0,0) -- (30:1) -- (-30:1) --(0,-1)--(0,0);
}

\newcommand\cube[3]{
  \topside{#1}{#2}{#3} \leftside{#1}{#2}{#3} \rightside{#1}{#2}{#3}
}

\newcommand\planepartition[1]{
 \setcounter{x}{-1}
  \foreach \a in {#1} {
    \addtocounter{x}{1}
    \setcounter{y}{-1}
    \foreach \b in \a {
      \addtocounter{y}{1}
      \setcounter{z}{-1}
      \foreach \c in {0,...,\b} {
        \addtocounter{z}{1}
      \ifthenelse{\c=0}{\setcounter{z}{-1},\addtocounter{y}{0}}{
        \cube{\value{x}}{\value{y}}{\value{z}}
      }
    }
  }
 }
}

\usetikzlibrary{arrows}
\usetikzlibrary{shapes}

\newtheorem{theorem}{Theorem}[section]
\newtheorem{proposition}[theorem]{Proposition}
\newtheorem{lemma}[theorem]{Lemma}
\newtheorem{deflemma}[theorem]{Definition-Lemma}

\newtheorem{corollary}[theorem]{Corollary}

\theoremstyle{definition}
\newtheorem{remark}[theorem]{Remark}
\newtheorem{definition}[theorem]{Definition}
\newtheorem{example}[theorem]{Example}

\DeclareMathOperator{\ess}{ess}
\DeclareMathOperator{\PD}{PD}
\DeclareMathOperator{\BPD}{BPD}
\DeclareMathOperator{\ASM}{ASM}
\DeclareMathOperator{\TSSCPP}{TSSCPP}
\DeclareMathOperator{\wt}{wt}
\DeclareMathOperator{\blnk}{blank}
\DeclareMathOperator{\cross}{cross}
\DeclareMathOperator{\cnt}{count}
\DeclareMathOperator{\nw}{NW}
\DeclareMathOperator{\slide}{\mathrm{Slide}}
\DeclareMathOperator{\droop}{\mathrm{Droop}}

\title{A pipe dream perspective on totally symmetric self-complementary plane partitions}

\author{Daoji Huang and Jessica Striker}

\begin{document}

\maketitle

\begin{abstract}   We characterize totally symmetric self-complementary plane partitions (TSSCPP) as bounded compatible sequences satisfying a Yamanouchi-like condition. As such, they are in bijection with certain pipe dreams.
  Using this characterization and 
 the recent bijection of [Gao-Huang] between reduced pipe dreams and reduced bumpless pipe dreams, we give a bijection between  alternating sign matrices and TSSCPP in the reduced, 1432-avoiding case. We also give a different bijection in the 1432- and 2143-avoiding case that preserves natural poset structures on the associated pipe dreams and bumpless pipe dreams.  \end{abstract}

\section{Introduction}
\emph{Plane partitions} are three-dimensional analogues of ordinary partitions. Just as partitions in an $a\times b$ are counted by a lovely formula $\binom{a+b}{a}$, plane partitions in an $a\times b\times c$ box are enumerated by MacMahon's product formula \scalebox{.8}{$\displaystyle\prod_{i=1}^a \displaystyle\prod_{j=1}^b \displaystyle\prod_{k=1}^c \displaystyle\frac{i+j+k-1}{i+j+k-2}$}~\cite{MacMahon}.
In a 1986~\cite{Stanley_PP}, Stanley considered symmetry operations on plane partitions, namely, reflection (transpose), rotation, and complementation. This yielded 10 symmetry classes of plane partitions consisting of plane partitions invariant under combinations of these operations. The plane partitions invariant under all three operations are called \emph{totally symmetric self-complementary} (\emph{TSSCPP}). As in the case of all plane partitions, each symmetry class has a nice enumeration. The set of TSSCPP inside a $2n\times 2n\times 2n$ box was shown in 1994 by Andrews~\cite{Andrews1994} to be counted by \scalebox{.9}{$\displaystyle\prod_{j=0}^{n-1}\displaystyle\frac{(3j+1)!}{(n+j)!}$}. This was, at the time, the conjectured~\cite{MRRASMDPP} number of $n\times n$ \emph{alternating sign matrices} (\emph{ASM}).
The 1996 proofs of this conjecture~\cite{ZEILASM,kuperbergASMpf} sparked a search for a natural, explicit bijection between TSSCPP and ASM. Partial bijections have been found on small subsets, including 
the permutation case \cite{PermTSSCPP}, the case of two monotone triangle diagonals \cite{Biane_Cheballah_1,Bettinelli}, and the $312$-avoiding case~\cite{Ayyer312}. 
This paper interprets TSSCPP as \emph{pipe dreams} to extend the bijection of \cite{PermTSSCPP} to what appears to be a larger subset than any previous partial bijection; see Section~\ref{sec:remarks} for discussion. 

Our first main theorem is below; see Figure~\ref{fig:main_bij} for an example and Section~\ref{sec:Background} for the relevant definitions. 
Given $\pi\in S_n$, let $\TSSCPP^{red}(\pi)$ denote the set of TSSCPP whose associated pipe dream is \emph{reduced} and has permutation $\pi$, and let $\ASM^{red}(\pi)$ denote the set of ASM whose associated \emph{bumpless pipe dream} is reduced and has permutation $\pi$. 
\begin{theorem}
\label{thm:main}
Let $\pi\in S_n$. There is an explicit weight-preserving injection $\varphi$ from $\TSSCPP^{red}(\pi)$ to $\ASM^{red}(\pi)$.
If $\pi$ avoids $1432$, then $\varphi$ is a bijection.
\end{theorem}

While the bijection of Theorem~\ref{thm:main} preserves a meaningful weight on both sides, it does not, in general, preserve the natural partial order. A corollary of our second main result, Theorem~\ref{thm:poset_bij}, 
gives a different poset-preserving bijection between $\TSSCPP^{red}(\pi)$ and $\ASM(\pi)$ in the case that $\pi$ avoids both $1432$ and $2143$. (Note in this case, $\ASM^{red}(\pi)=\ASM(\pi)$.) Theorem~\ref{thm:poset_bij} itself relates the posets  $\slide(\pi)$ on pipe dreams and $\droop(\pi)$ on bumpless pipe dreams of such permutations, giving a poset-preserving bijection by decomposing into  Grassmannian and inverse-Grassmannian blocks. 

The paper is organized as follows. Section~\ref{sec:Background} contains background on the relevant objects, including the permutation case TSSCPP bijection of \cite{PermTSSCPP} and the bijection of \cite{GH} between reduced pipe dreams and reduced bumpless pipe dreams, which are  important ingredients in our proof of Theorem~\ref{thm:main}. Section~\ref{sec:TSSCPP_PD} proves Theorem~\ref{thm:TSSCPP_Yam} characterizing TSSCPP as pipe dreams subject to a Yamanouchi-like condition. Section~\ref{sec:main} concerns Theorem~\ref{thm:main} and its proof. Section~\ref{sec:poset} proves Theorems~\ref{thm:invGrass}, \ref{thm:Grass}, and \ref{thm:poset_bij} relating the posets $\droop(\pi)$ and $\slide(\pi)$ in the respective cases where $\pi$ is inverse-Grassmannian, Grassmannian, or avoiding both $1432$ and $2143$. These theorems yield Corollaries~\ref{cor:invGrass}, \ref{cor:Grass}, and~\ref{cor:ASMTSSCPPposet_bij}, which give poset-preserving bijections between $\TSSCPP^{red}(\pi)$ and $\ASM(\pi)$ for these three types of permutations. Section~\ref{sec:remarks} gives some concluding remarks.

An extended abstract of this paper was published in the proceedings of the 2023 FPSAC conference~\cite{BPD_FPSAC2023}.

\begin{figure}
\begin{center}
\scalebox{.8}{\includegraphics[scale=.35]{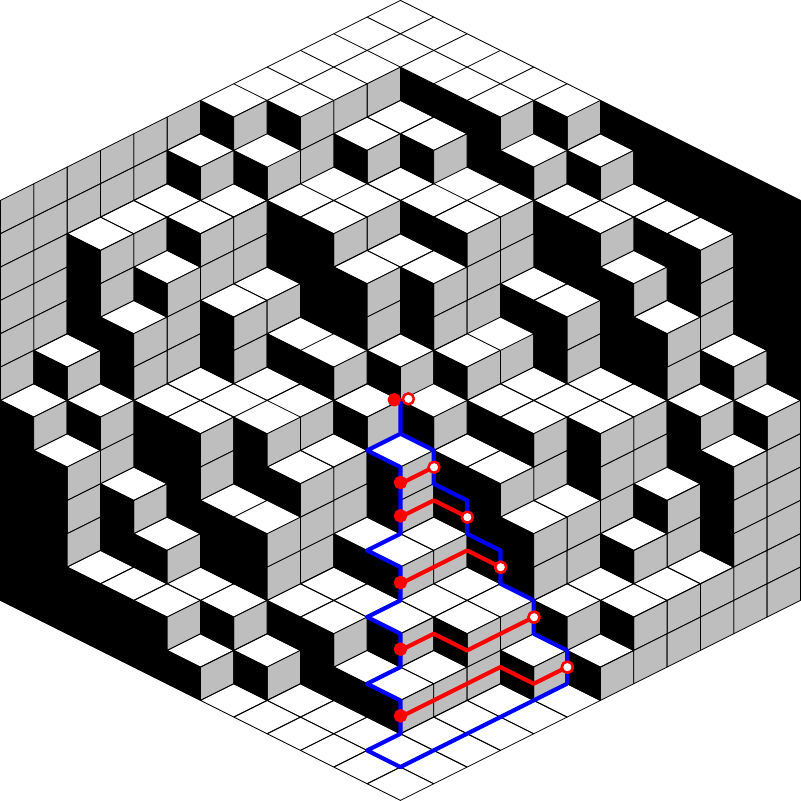} \raisebox{1.5cm}{$\Leftrightarrow$}\includegraphics[height=1.4in]{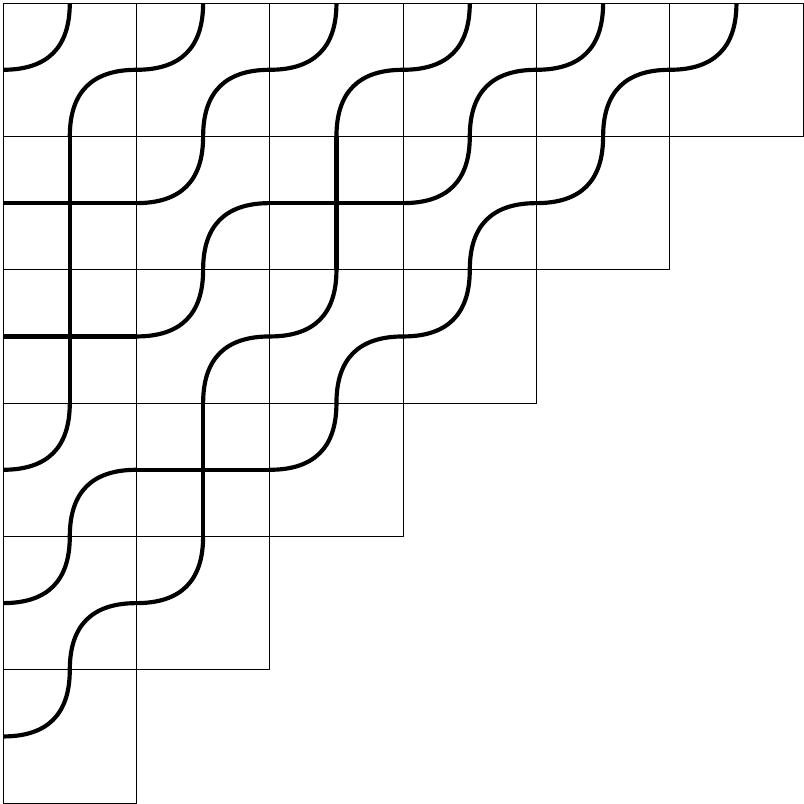} \raisebox{1.5cm}{$\Leftrightarrow$}
\includegraphics[height=1.4in]{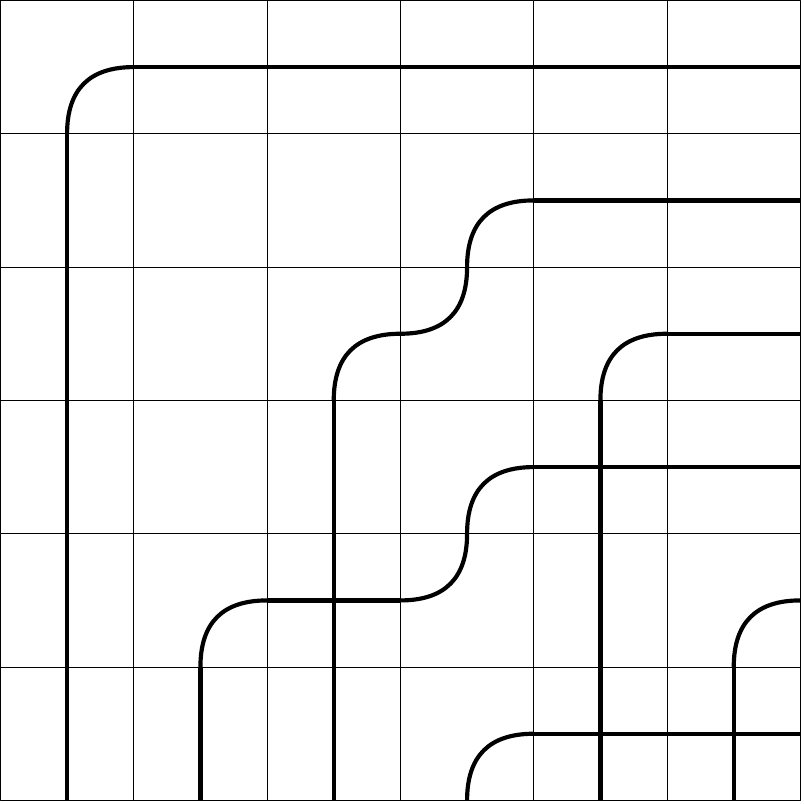} \raisebox{1.2cm}{$\Leftrightarrow$}
\raisebox{1.5cm}{$\left( \begin{array}{rrrrrr}
1&0&0&0&0&0\\
0&0&0&1&0&0\\
0&0&1&-1&1&0\\
0&0&0&1&0&0\\
0&1&0&-1&0&1\\
0&0&0&1&0&0
\end{array}\right)$}}
\end{center}
\caption{An example of the bijection of this paper. From left to right the objects are: TSSCPP, pipe dream, bumpless pipe dream, ASM. The pipe dream and bumpless pipe dream both have permutation $135264$, which avoids $1432$. Note the black rhombi in column $k$ (from the left) of the TSSCPP fundamental domain correspond to the cross tiles in row $k$ (from the top) of the pipe dream. This equals the number of blank tiles in row $k$ of the bumpless pipe dream, which is the number of positive inversions of row $k$ of the ASM.}
\label{fig:main_bij}
\end{figure}

\section{Background}
\label{sec:Background}
In this section, we review relevant definitions and bijections from the literature. Subsections~\ref{sec:asm}, \ref{sec:tsscpp}, \ref{sec:bpd}, and \ref{sec:pd} review definitions of ASM, TSSCPP, bumpless pipe dreams, and pipe dreams, respectively. Subsections~\ref{sec:boolean} and \ref{sec:BPDPD_bij} contain less-familiar bijections that are important for our main results.

\subsection{Alternating sign matrices}
\label{sec:asm}
In this subsection, we define alternating sign matrices (see e.g.~\cite{MRRASMDPP}) and the weight that is preserved in Theorem~\ref{thm:main}.
\begin{definition}
An \textbf{alternating sign matrix} (\textbf{ASM}) is a square matrix with entries in $\{0,1,-1\}$ such that the rows and columns each sum to $1$ and the nonzero entries alternate in sign across each row and across each column.
\end{definition}
Alternating sign matrices are in bijection with configurations of the  \emph{six-vertex / square ice model} of statistical physics with \emph{domain wall boundary conditions}; this was an essential element of the enumeration proof of~\cite{kuperbergASMpf}. 
The $3\times 3$ alternating sign matrices are below.
\[
\footnotesize
\left( 
\begin{array}{rrr}
1 & 0 & 0 \\
0 & 1 & 0\\
0 & 0 & 1
\end{array} \right)
\left( 
\begin{array}{rrr}
1 & 0 & 0 \\
0 & 0 & 1\\
0 & 1 & 0
\end{array} \right)
\left( 
\begin{array}{rrr}
0 & 1 & 0 \\
0 & 0 & 0\\
0 & 0 & 1
\end{array} \right)
\left( 
\begin{array}{rrr}
0 & 1 & 0 \\
1 & -1 & 1\\
0 & 1 & 0
\end{array} \right)
\left( 
\begin{array}{rrr}
0 & 1 & 0 \\
0 & 0 & 1\\
1 & 0 & 0
\end{array} \right)
\left( 
\begin{array}{rrr}
0 & 0 & 1 \\
1 & 0 & 0\\
0 & 1 & 0
\end{array} \right)
\left( 
\begin{array}{rrr}
0 & 0 & 1 \\
0 & 1 & 0\\
1 & 0 & 0
\end{array} \right)
\]
\normalsize

In Figure~\ref{fig:ASMcorresp}, left and center-left are an alternating sign matrix and its corresponding square ice configuration; the horizontal molecules correspond to $+1$, the vertical molecules correspond to $-1$, and all other molecules correspond to $0$. Center-right is its six-vertex configuration, where the six molecule configurations are replaced by directed edges.  Figure~\ref{fig:ASMcorresp}, right, shows the corresponding \emph{bumpless pipe dream}, which will be discussed shortly.

\begin{figure}[htbp]
\begin{center}
%\begin{tabular}[t]{lcccccr}
\raisebox{.45in}{$\left( 
\begin{array}{rrrr}
0 & 1 & 0 & 0 \\
1 & -1 & 0 & 1\\
0 & 0 & 1 & 0\\
0 & 1 & 0 & 0
\end{array} \right)$} 
\scalebox{.85}{
\hspace{.25in} \includegraphics[height=1.2in]{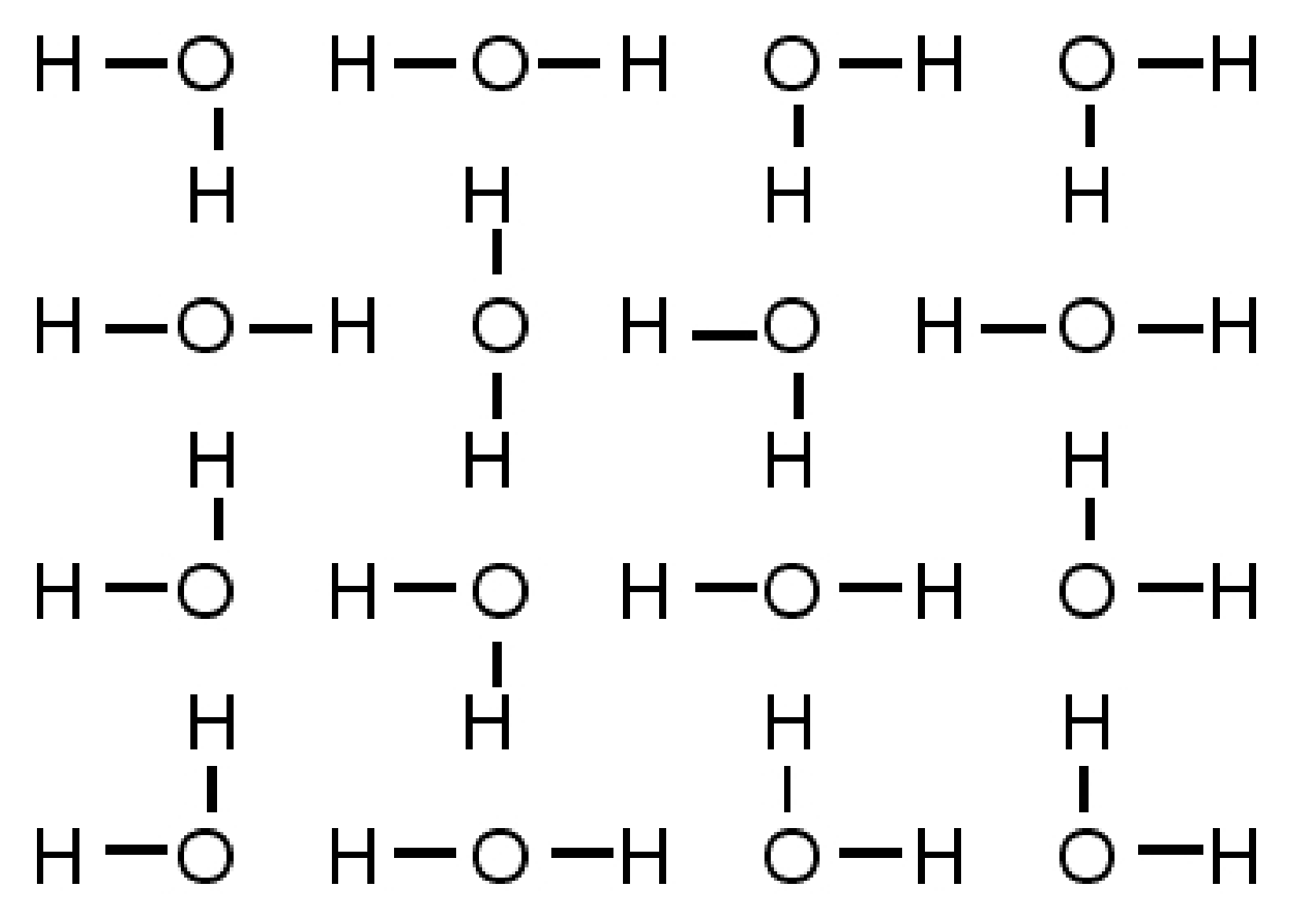} \hspace{.35in}\includegraphics[height=1.2in]{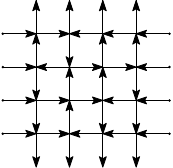}\hspace{.35in}\includegraphics[height=1.2in]{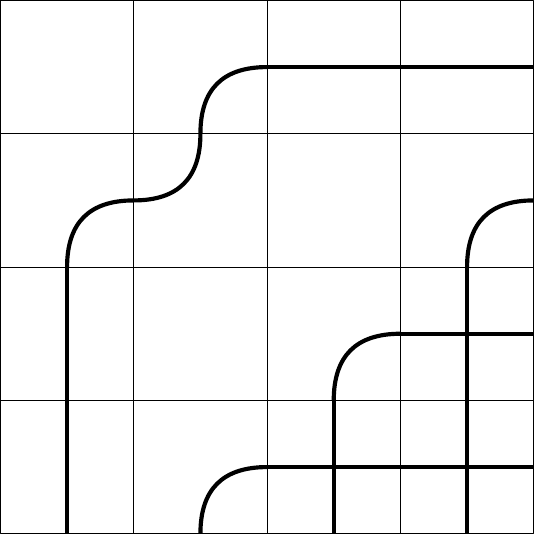}
}
\end{center}
\caption{An alternating sign matrix and its corresponding square ice configuration, six-vertex configuration, and bumpless pipe dream.}
\label{fig:ASMcorresp}
\end{figure}

An important statistic on an alternating sign matrix $A$ is the \textbf{positive inversion number}:
\[\nu(A) = \sum_{1\leq i< k<n}\sum_{1\leq \ell\leq j\leq n}A_{ij}A_{k\ell}.
\]
The positive inversion number of $A$ equals the number of \includegraphics[scale = .5]{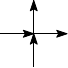} vertices in the corresponding six-vertex configuration. Equivalently, $\nu(A)$ equals the number of entries $A_{ij}=0$ such that $\sum_{i'=1}^i A_{i'j} = \sum_{j'=1}^j A_{ij'} = 0$. Let $\nw(A)$ denote the set of matrix indices $(i,j)$ of all such entries. We use the notation $\nw(A)$ since these $0$ entries are precisely those whose first nonzero entry to their right is a $1$ and the first nonzero entry below is a $1$; that is, they are north and west of $1$ entries. See for instance \cite[Sections 2.1 and 5.2]{BEHRENDMultiplyRefinedASM} for these well-known correspondences. The \textbf{weight} that is preserved in our main bijection is the following refinement of the positive inversion number statistic by row: 
 $\wt(A):=\prod_{(i,j)\in\nw(A)}x_i$. For the ASM in Figure~\ref{fig:ASMcorresp}, $\wt(A)=x_1x_2x_3$.

\subsection{Totally symmetric self-complementary plane partitions}
\label{sec:tsscpp}
In this subsection, we define plane partitions and their symmetry classes (see e.g.~\cite{Stanley_PP}). The weight that is preserved in Theorem~\ref{thm:main} will be discussed in Subsection~\ref{sec:boolean}.
\begin{definition}
A \textbf{plane partition} $t$
is a rectangular array of nonnegative integers $(t_{i,j})_{i,j\geq 1}$ such that 
$t_{i,j}\geq t_{i',j'}$ if $i\leq i'$, $j\leq j'$. We say $t$ is in an $a\times b\times c$ \textbf{bounding box} if $t_{i,j}=0$ whenever $i>a$ or $j>b$ and $t_{i,j}\leq c$ for all $i,j$. Let $PP(a\times b\times c)$ denote the set of plane partitions in an $a\times b\times c$ bounding box.
\end{definition}
\begin{remark}
We can also view $t\in PP(a\times b\times c)$ as a finite set of positive integer lattice points $(i, j, k)$ with $1\leq i\leq a$, $1\leq j\leq b$, and $1\leq k\leq c$ such
that if $(i, j, k) \in t$ and $1 \leq i'\leq i$, $1 \leq j'\leq j$,  $1 \leq k'\leq k$ then $(i',j',k')\in t$. 
This well-known bijection is given as $(i,j,k)\in t$ if and only if $t_{i,j}\geq k$. We will use both of these characterizations in the next definition.
\end{remark}

\begin{definition}
A plane partition $t$ is \textbf{symmetric} if $t_{i,j}=t_{j,i}$ for all $i,j$.  
$t$ is \textbf{cyclically
symmetric} if whenever $(i, j, k) \in t$ then $(j, k, i) \in t$ as well. $t$ is \textbf{totally symmetric} if it is both symmetric and cyclically symmetric, so that whenever $(i, j, k) \in t$ then all six permutations of $(i, j, k)$ are also in $t$.
The \textbf{complement} $t^C$ of $t$ inside a given bounding box $a \times b \times c$ is defined as $t^C_{i,j}=c-t_{a-i+1,b-j+1}$ for all $1\leq i\leq a, 1\leq j\leq b$. That is, $t^C_{i,j}$ equals the number of empty cubes above $t_{a-i+1,b-j+1}$ in the bounding box.
A plane partition $t$ is \textbf{self-complementary} inside a given bounding box if $t=t^C$. A \textbf{totally symmetric self-complementary plane partition (TSSCPP)} 
is a plane partition which is both totally symmetric and self-complementary.
\end{definition}

Note that for there to exist a self-complementary plane partition in an $a\times b\times c$ bounding box, the volume $abc$ of the box  must be an even number. In addition, cyclic symmetry requires $a=b=c$, therefore, we need $a=b=c=2n$ for there to exist a TSSCPP inside an $a\times b\times c$ bounding box. 
\begin{definition}
Let $\TSSCPP(n)$ denote the set of TSSCPP inside a $2n\times 2n\times 2n$  box.
\end{definition}

\subsection{TSSCPP boolean triangles and a permutation case bijection}
\label{sec:boolean}
In this subsection, we review the characterization from~\cite{PermTSSCPP} of TSSCPP as boolean triangles  and the bijection of the same paper between permutation matrices and TSSCPP boolean triangles whose entries weakly decrease along rows. We also describe the weight on TSSCPP preserved in Theorem~\ref{thm:main}.

\begin{definition}[Def 2.12 of \cite{PermTSSCPP}]
\label{def:tsscppbool}
A \textbf{TSSCPP boolean triangle} of order $n$ is a triangular integer array $b=\{b_{i,j}\}$ for $1 \leq i \leq n - 1$, $n - i \leq j \leq n - 1$ with entries in $\{0, 1\}$ such that the diagonal partial sums satisfy the following inequality for all $1\leq j< i\leq n-1$:
\begin{equation}
\label{eq:TSSCPPineq}
1 + \sum_{k=j+1}^{i} b_{k,n-j-1}\geq \sum_{k=j}^{i} b_{k,n-j}.
\end{equation}
Call this the $(i,j)$-inequality, in which $n-j$ and $n-j-1$ are the diagonals being compared and $i$ indicates the row index of where the sums stop.
\end{definition}

We give below the indexing of a generic TSSCPP boolean triangle.
\begin{center}
\scalebox{.85}{$\begin{array}{ccccccccccc}
  & & & & & b_{1,n-1} & & & & & \\
  & & & & b_{2,n-2} & & b_{2,n-1} & & & & \\
  & & & b_{3,n-3} & & b_{3,n-2} & & b_{3,n-1} & & & \\
  &  & &  & & \vdots & &  & &  &\\
  &  & b_{n-1,1} & & b_{n-1,2} & \cdots & b_{n-1,n-2} & & b_{n-1,n-1} & &
\end{array}$}
\end{center}

Below are a non-example and an example of a TSSCPP boolean triangle.
\begin{center}
\scalebox{.85}{$\begin{array}{ccccccccccc}
  & & & & & 1 & & & & & \\
  & & & & 1 & & 1 & & & & \\
  & & & 1 & & 0 & & 0 & & & \\
  & & 1 & & 0 & & 0 & & 1 & & \\
  & 0 & & 0 & & 0 & & 1 & & 0 &\end{array}$
  \hspace{1cm}
  $\begin{array}{ccccccccccc}
  & & & & & 1 & & & & & \\
  & & & & 1 & & 1 & & & & \\
  & & & 1 & & 0 & & 0 & & & \\
  & & 1 & & 0 & & 1 & & 1 & & \\
  & 0 & & 0 & & 0 & & 0 & & 0 &\end{array}$}
\end{center}
In the left triangle, the $(4,1)$-inequality is not satisfied, since $\sum_{k=1}^{4} b_{k,n-1} = 3$ while $\sum_{k=2}^{4} b_{k,n-2} = 1$. In the triangle on the right, all $(i,j)$-inequalities are satisfied.

\begin{proposition}[Prop 2.13 of \cite{PermTSSCPP}]
\label{prop:bool_bij}
TSSCPP boolean triangles of order $n$ are in bijection with $\TSSCPP(n)$. 
\end{proposition}

The bijection proceeds by taking the fundamental domain of the TSSCPP, transforming it into a nest of non-intersecting lattice paths, and then recording the two different types of steps in each path as $0$ and $1$. The diagonal partial sum condition (\ref{eq:TSSCPPineq}) is equivalent to the requirement that the paths do not intersect. See~\cite{PermTSSCPP} for details.

We now review the characterization of a certain subset of TSSCPP boolean triangles. 
\begin{definition}[Def 3.1 of \cite{PermTSSCPP}]
A \textbf{permutation TSSCPP boolean triangle} is a TSSCPP boolean triangle with weakly decreasing rows.
\end{definition}

That is, the entries equal to one in a permutation TSSCPP boolean triangle are all left-justified.
The terminology `permutation' in the above definition is justified by the weight-preserving bijection in the theorem below. An example of this bijection is given in Figure~\ref{fig:permbij}.

\begin{theorem}[Theorem 3.5 of \cite{PermTSSCPP}]
There is a natural, statistic-preserving bijection  
between $n\times n$ permutation matrices with inversion number $p$ 
%whose one in the last row is in column $k$ and whose one in the last column is in row $\ell$ 
and permutation TSSCPP boolean triangles of order $n$ with $p$ zeros.
%exactly $n - k$ of which are contained in the last row, and for which the lowest one in diagonal $n - 1$ is in row $\ell - 1$.
\end{theorem}

\begin{figure}[htbp]
\begin{center}
$\begin{array}{c}
\mbox{TSSCPP}\\
\includegraphics[scale=.3]{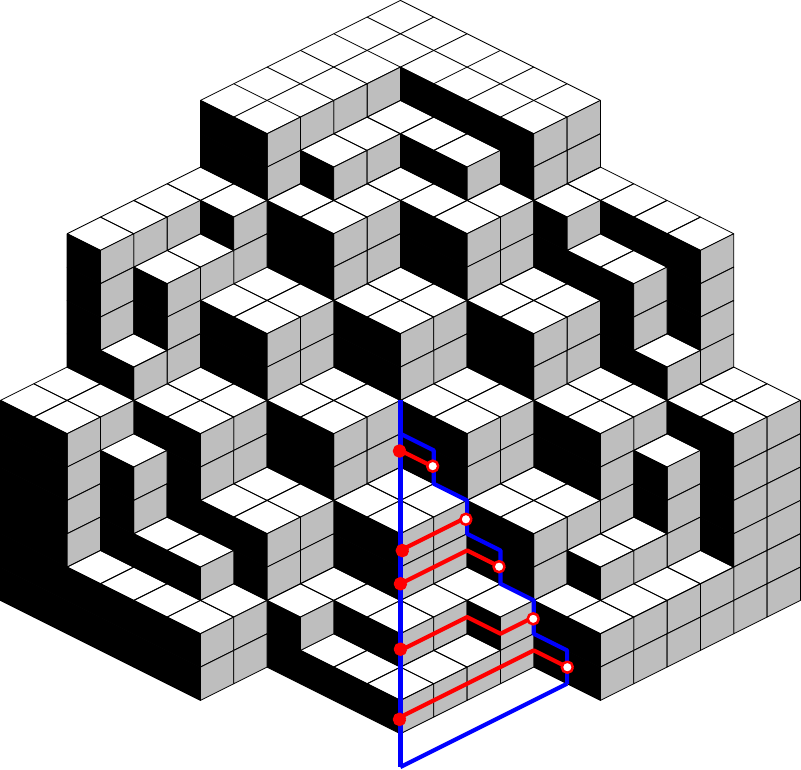}
\end{array}$
\raisebox{0cm}{
$\leftrightarrow$}
$\begin{array}{c}
\mbox{Permutation }\\
\mbox{TSSCPP}\\
\mbox{Boolean triangle}\\
\begin{array}{ccccccccccc}
  & & & & & 1 & & & & & \\
  & & & & 0 & & 0 & & & & \\
  & & & 1 & & 1 & & 0 & & & \\
  & & 0 & & 0 & & 0 & & 0 & & \\
  & 1 & & 0 & & 0 & & 0 & & 0 &\end{array} \end{array}$
\raisebox{0cm}{
$\leftrightarrow$}
$\begin{array}{c}
\mbox{Permutation}\\
\mbox{Matrix}\\
\mbox{ }\\
\scalebox{.9}{$\left( \begin{array}{cccccc}
0&0&0&1&0&0\\
0&0&0&0&0&1\\
0&0&1&0&0&0\\
0&0&0&0&1&0\\
1&0&0&0&0&0\\
0&1&0&0&0&0
\end{array}\right)$} \end{array}$
\end{center}
\caption{An example of the permutation case bijection of \cite[Theorem 3.5]{PermTSSCPP}}
\label{fig:permbij}
\end{figure}

The injection $\varphi$ in Theorem~\ref{thm:main} extends this bijection under the mild transformation of flipping the resulting matrix vertically (or reversing the one-line notation of the permutation). Thus,  in Theorem~\ref{thm:main} we  instead map the TSSCPP in Figure~\ref{fig:permbij}, left, to the vertical reflection of the matrix in Figure~\ref{fig:permbij}, right.

The TSSCPP \textbf{weight}  preserved in Theorem~\ref{thm:main} is the number of $1$ entries in the $k$th row from the bottom in its boolean triangle. More specifically, for $T\in\TSSCPP(n)$ with boolean triangle $b$, let $\text{one}(T)$ denote the set of indices of the entries of $b$ that equal $1$. Then
 $\wt(T):=\prod_{(i,j)\in \text{one}(T)}x_{n-i}$. For the TSSCPP in Figure~\ref{fig:permbij}, $\wt(T)=x_1x_3^2x_5$ and for the TSSCPP in Figure~\ref{fig:topPDex}, 
 $\wt(T)=x_2^2x_3x_4$. Note the weight can also be seen directly on the TSSCPP fundamental domain, as each $1$ in row $n-k$ of $b$ corresponds to a black rhombus in column $k$ (from the left) of the fundamental domain, as shown in these figures.
 
\subsection{Bumpless pipe dreams}
\label{sec:bpd}
In this subsection, we define bumpless pipe dreams and describe the bijection with alternating sign matrices.
\begin{definition}
A \textbf{bumpless pipe dream} \cite{LLS} of size $n$ is a tiling of an $n\times n$ grid of squares by the following six types of tiles:
$\+$, $\htile$, $\vtile$, $\rt$, $\jt$, $\bl$,
such that $n$ pipes traveling from the south border to the east border are formed. 
We denote the set of bumpless pipe dreams of size $n$ as $\BPD(n)$.
We say a bumpless pipe dream is \textbf{reduced} if no two pipes cross twice. We associate a permutation to each reduced bumpless pipe dream by labeling the pipes $1,\cdots, n$ from left to right on the south border and read off the pipe labels from top to bottom on the east border. Let $\BPD^{red}(\pi)$ denote the set of all reduced bumpless pipe dreams with permutation $\pi$.
\end{definition}

A \textbf{simple droop} is a  move on a bumpless pipe dream bounded by a $2\times 2$ square, as shown below. The four pairs of $2\times 2$ squares show all four possibilities a pipe enters and leaves the $2\times 2$ square.
\begin{center}
\includegraphics[scale=0.8]{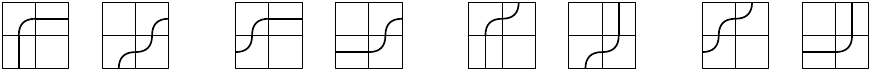}
\end{center}

For each reduced bumpless pipe dream $D$, we define its \textbf{weight} to be the monomial  $\wt(D):=\prod_{(i,j)\in\blnk(D)}x_i$ where $\blnk(D)$ denotes the set of the $\bl$-tiles in $D$. 

There is a natural, weight-preserving bijection between $\BPD(n)$ and $\ASM(n)$, as described in \cite{weigandt}. To obtain an ASM from a BPD we simply replace each $\rt$ with a 1, each $\jt$ with a $-1$, and the other four allowed squares with 0s. The blank tiles $\bl$ of a BPD $D$ correspond to the NW zeros of the associated ASM $A$. Thus $\blnk(D)=NW(A)$, so $\wt(D)=\wt(A)$. For the inverse map, it is not difficult to see the positions of $\rt$ and $\jt$ uniquely determine a bumpless pipe dream. See Figure~\ref{fig:ASMcorresp} for an example.

\subsection{Pipe dreams}
\label{sec:pd}
In this subsection, we define pipe dreams and bounded compatible sequences and describe the bijection between them.
\begin{definition} 
A \textbf{pipe dream} \cite{BB} of size $n$ is a tiling of an $n\times n$ grid of squares 
with two kinds of tiles, the cross-tile $\+$ and elbow-tile $\bt$, such that the positions on or below the main (anti)diagonal only consist of elbow-tiles. We think of a pipe dream as $n$ pipes, labelled $1,\cdots, n$ traveling from the north border and exiting from the west border. We denote the set of pipe dreams of size $n$ as $\PD(n)$. We say a pipe dream is \textbf{reduced} if no two pipes cross twice. We associate a permutation to each reduced pipe dream by reading from top to bottom the labels of each pipe on the west border of the pipe dream. (One can also assign permutations to non-reduced pipe dreams, but this will not be important for the present paper.) Let $\PD^{red}(\pi)$ denote the set of reduced pipe dreams with permutation $\pi$.
\end{definition}

The set of pipe dreams for a fixed permutations are connected by \emph{chute} and \emph{ladder} moves. For precise definitions see \cite{BB}. When a ladder (or chute) move is bounded by a $2\times 2$ square, we call this move a \textbf{simple slide}, as shown below.
\begin{center}
\includegraphics[scale=0.7]{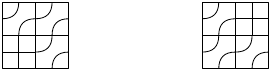}
\end{center}

Figure~\ref{fig:PD1432} shows the set of pipe dreams $\PD^{red}(1432)$.
The first four pipe dreams are connected by simple slides; the fifth is not.

\begin{figure}[h]
\centering
\includegraphics[scale=0.7]{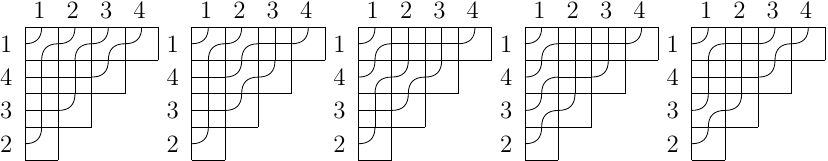}
\caption{$\PD^{red}(1432)$}
\label{fig:PD1432}
\end{figure}

For each 
pipe dream $D$, we define its \textbf{weight} to be the monomial  $\wt(D):=\prod_{(i,j)\in\cross(D)}x_i$ where $\cross(D)$ denotes the set of the $\+$-tiles in $D$.

\begin{definition}
A \textbf{bounded compatible sequence} \cite{BJS} is a pair $(\mathbf{a},\mathbf{r})$ where $\mathbf{a}=(a_1,\cdots, a_{\ell})$ and $\mathbf{r}=(r_1,\cdots, r_\ell)$ are words of positive integers, satisfying the following conditions:
\begin{enumerate}
\item[(a)] $r_1\ge r_2\ge \cdots\ge r_\ell$,
\item[(b)] $a_i\ge r_i$ for all $1\le i\le \ell$,
\item[(c)] $r_i>r_{i+1}$ if $a_i\ge a_{i+1}$.
\end{enumerate}
\end{definition}
There is a simple bijection between $\PD(n)$ and the set of all bounded compatible sequences where $a_i<n$ for each $i$, see \cite{BB}. Given a bounded compatible sequence $(\mathbf{a},\mathbf{r})$, we may construct a pipe dream by putting a cross-tile at position $(r_i, a_i+1-r_i)$ for each $1\le i\le \ell$ and fill the remaining positions with elbow-tiles. Conversely, given a pipe dream, we may construct a bounded compatible sequence as follows: scan the pipe dream from bottom to top and within each row left to right, and whenever we encounter a cross-tile at position $(r,c)$ we append $(r+c-1,r)$ to the compatible sequence. For example, the corresponding bounded compatible sequences for the pipe dreams in Figure~\ref{fig:PD1432} are as follows; the vector $\mathbf{a}$ is recorded in the top row and $\mathbf{r}$ in the bottom row.
\[\left(\begin{smallmatrix}3 & 2 & 3\\
3 & 2 & 2\end{smallmatrix}\right),
\left(\begin{smallmatrix}3 & 2 & 3\\
3 & 2 & 1\end{smallmatrix}\right),
\left(\begin{smallmatrix}3 & 2 & 3\\
3 & 1 & 1\end{smallmatrix}\right),
\left(\begin{smallmatrix}3 & 2 & 3\\
2 & 1 & 1\end{smallmatrix}\right),
\left(\begin{smallmatrix}2 & 3 & 2\\
2 & 2 & 1\end{smallmatrix}\right).\]

\subsection{Reduced BPD-PD bijection} 
\label{sec:BPDPD_bij}
Both reduced pipe dreams and reduced bumpless pipe dreams give combinatorial formulas for Schubert polynomials $\mathfrak{S}_\pi$, $\pi\in S_\infty$ which are important polynomials in the study of Schubert calculus \cite{BB,LLS}. Explicitly,
\[\mathfrak{S}_\pi=\sum_{D\in\PD^{red}(\pi)}\prod_{(r,c)\in\operatorname{cross}(D)}x_r=\sum_{D\in\BPD^{red}(\pi)}\prod_{(r,c)\in\operatorname{blank}(D)}x_r.\]
For this reason, there exists a weight-preserving bijection  between $\PD^{red}(\pi)$ and $\BPD^{red}(\pi)$, where the weight of a PD or BPD is its monomial contribution to the Schubert polynomial indexed by its permutation.

In \cite{GH}, such an explicit direct bijection $\varphi:\BPD^{red}(\pi)\rightarrow\PD^{red}(\pi)$ is given using an iterative algorithm. To find the image of a BPD under $\varphi$, the algorithm computes in each iteration the position of one crossing in the corresponding PD. For a detailed description of this process, see \cite[Definition 3.1]{GH}. For explicit examples, see \cite[Example 3.4]{GH}. This bijection is weight-preserving; in particular, for $D\in\BPD^{red}(\pi)$, the number of blank tiles in row $k$ equals the number of cross-tiles in row $k$ of $\varphi(D)$.

Because the bijection is weight-preserving and there is a unique lowest weight monomial that corresponds to the Lehmer code of the permutation in each Schubert polynomial, the permutation BPD is mapped to the \textbf{bottom pipe dream}, the unique pipe dream with all crosses left-justified. 
\section{Characterizing TSSCPP as pseudo-Yamanouchi pipe dreams}
\label{sec:TSSCPP_PD}
This section focuses on our first theorem: a characterization of TSSCPP as a subset of all (reduced and non-reduced) pipe dreams.
\subsection{Mapping TSSCPP into pipe dreams}
\label{subsec:map}
Recall the bijection of Proposition~\ref{prop:bool_bij} from TSSCPP to the TSSCPP boolean triangles of Definition~\ref{def:tsscppbool}.
As TSSCPP boolean triangles are triangular arrays with entries in $\{0,1\}$, we can transform them to pipe dreams (reduced and non-reduced), since these are triangular arrays of tiles with two choices for each spot. There are several possibilities for how to do this; we choose to correspond each $1$ to a cross-tile $\+$  and each $0$ to an elbow-tile $\bt$. There are also several choices for orientation of the triangle. We set the following convention.

Given a TSSCPP boolean triangle $b$ of order $n$, we create a triangular array  $y_{i,j}$, $1\leq i\leq n-1$, $1\leq j\leq n-i$ of zeros and ones where $y_{i,j}=b_{n-i,i+j-1}$. That is, 
we flip $b$ vertically and justify to the left.

\begin{center}
$\begin{array}{ccccc}
  b_{n-1,1} & b_{n-1,2} & b_{n-1,3} & \cdots & b_{n-1,n-1}  \\
  b_{n-2,2} & b_{n-2,3} & \cdots & b_{n-2,n-1}  \\
  b_{n-3,3} & \cdots & b_{n-3,n-1} \\
  &   \iddots\\
b_{1,n-1}
\end{array}
\begin{array}{ccccc}
  y_{1,1} & y_{1,2} & y_{1,3} & \cdots & y_{1,n-1}  \\
  y_{2,1} & y_{2,2} & \cdots & y_{2,n-2} & \\
  y_{3,1} & \cdots & y_{3,n-3} &  &   \\
  &   \iddots & &  & \\
  y_{n-1,1}
\end{array}$
\end{center}

The inequality of Definition~\ref{def:tsscppbool} translates to the following:
\[1 + \sum_{k=1}^i y_{j-k,k}\geq \sum_{k=1}^{i+1} y_{j-k+1,k} \text{ \ \ for all }1\leq i< j\leq n-1.\]

Now we turn each $1$ into a cross-tile $\+$ and each $0$ into an elbow-tile $\bt$. We call the pipe dreams that lie in this image the TSSCPP pipe dreams. Note that permutation TSSCPP boolean triangles have weakly decreasing rows; this  corresponds to left-justified crosses in the associated pipe dream.

\begin{figure}[hbtp]
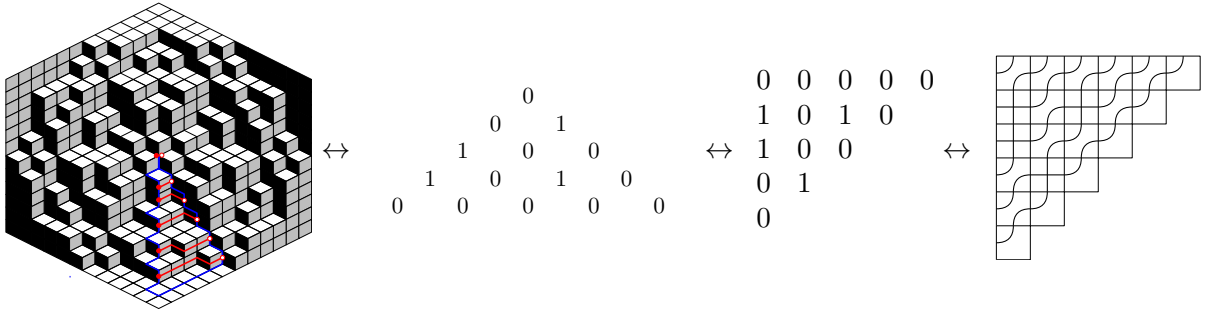

\begin{center}
\includegraphics[scale=.3]{TnonpermFundDomain} \raisebox{2cm}{$\leftrightarrow$}
\raisebox{2cm}{$
\scalebox{.8}{$\begin{array}{ccccccccccc}
  & & & & & 0 & & & & & \\
  & & & & 0 & & 1 & & & & \\
  & & & 1 & & 0 & & 0 & & & \\
  & & 1 & & 0 & & 1 & & 0 & & \\
  & 0 & & 0 & & 0 & & 0 & & 0 &\end{array}$} \leftrightarrow
  \hskip 0.5em
\begin{matrix}
0 & 0 & 0 & 0 & 0 \\
1 & 0 & 1 & 0 &       \\
1 & 0 & 0 &      &       \\
0 & 1     &       &      &       \\
0 &      &       &      &
\end{matrix}
\leftrightarrow$}
\hskip 0.5em
 \raisebox{.65cm}{\includegraphics[scale=.2]{TSSCPP_pipe_dream_fixed}}
\end{center}
\caption{An example of transforming a TSSCPP to a pipe dream. Note the weight of this TSSCPP is $x_2^2x_3x_4$.} 
\label{fig:topPDex}
\end{figure}

\subsection{A Yamanouchi-like condition on bounded compatible sequences}
In this subsection, we prove Theorem~\ref{thm:TSSCPP_Yam} characterizing TSSCPP pipe dreams. We also prove Lemma~\ref{lem:bottom_Yam}, which will be used in Section~\ref{sec:main}.
\begin{definition}
Given a bounded compatible sequence \[(\mathbf{a},\mathbf{r})=((a_1,\cdots, a_\ell), (r_1,\cdots, r_\ell)),\] define $\cnt(k,j)(\mathbf{a})$ (or $\cnt(k,j)$ when $\mathbf{a}$ is understood) to be the number of $j$ that appear in $a_1,\cdots, a_k$. We say that $(\mathbf{a}, \mathbf{r})$ is \textbf{pseudo-Yamanouchi} if for all $1\le k\le \ell$, $1\le j\le n-2$, $1+\cnt(k,j)\ge \cnt(k,j+1)$. We also say that a pipe dream is pseudo-Yamanouchi if its corresponding bounded compatible sequence is so.
\end{definition}

\begin{lemma}
\label{lem:bottom_Yam}
For any $\pi$, the bottom pipe dream is pseudo-Yamanouchi.
\end{lemma}
\begin{proof}
The bottom pipe dream is the unique pipe dream of $\pi$ with all left-justified cross-tiles. Thus the bounded compatible sequence $(\mathbf{a},\mathbf{r})$ is either empty or $\mathbf{a}$ is made up of increasing runs such that it can be written for some $m\geq 1$ as \[\mathbf{a}=(j_1,j_1+1,\ldots, j_1^*-1, j_1^*, j_2, j_2+1,\ldots,j_2^*-1, j_2^*, \ldots, j_{m}, j_{m}+1, \ldots, j_m^*-1, j_{m}^*)\] where $j_1>j_2>\cdots >j_{m}$ and $j_i^*\geq j_i$ for all $1\leq i\leq m$. Because $j_1>j_2>\cdots >j_{m}$, each increasing run needs to start with a smaller number than the previous. 

Suppose $(\mathbf{a},\mathbf{r})$ is not pseudo-Yamanouchi. Choose the smallest $k$ such that there exists a $j$ for which $1+\cnt(k,j) < \cnt(k,j+1)$. Since $\cnt(k,j)$ is a non-negative increasing function of $k$, it must be that  $1+\cnt(k-1,j) = \cnt(k-1,j+1)$ and $a_k=j+1$, since we chose $k$ to be the smallest value with the property. If we are in row $j+1$ ($r_k = j+1$), then this is the first time that $j+1$ has appeared in $\mathbf{a}$, so $\cnt(k,j+1) = 1$ and thus cannot be greater than $1+\cnt(k,j)$. If $r_k>j+1$, then $a_{k-1}=j$, since the cross-tiles are left-justified, and thus $\cnt(k-1,j) + 1 =\cnt(k,j)$. But $1+\cnt(k-1,j) \ge \cnt(k-1,j+1) = \cnt(k,j+1)-1$. So finally, $\cnt(k,j) \ge \cnt(k,j+1)-1$, which is a contradiction.
\end{proof}

\begin{theorem}
\label{thm:TSSCPP_Yam}
$\TSSCPP(n)$ is in weight-preserving bijection with the set of pseudo-Yamanouchi pipe dreams in $\PD(n)$.
\end{theorem}
\begin{proof}
We identify a TSSCPP with the 0-1 triangular array $(y_{i,j})_{1\le i\le n-1, 1\le j\le n-i}$ satisfying the inequalities $1 + \sum_{k=1}^i y_{j-k,k}\geq \sum_{k=1}^{i+1} y_{j-k+1,k}$
for all $1\leq i< j\leq n-1$, as described in Section~\ref{subsec:map}. These inequalities mean the following: for any position $(i,j)$ in the corresponding pipe dream, the number of crosses in the same diagonal as $(i,j)$ at or below row $i$ can be at most one more than the number of crosses in the previous diagonal at or below row $i$. Therefore it suffices to check this property when $(i,j)$ is a cross to decide whether the pipe dream is TSSCPP.

Now suppose $(a,r)$ is an entry of a pseudo-Yamanouchi compatible sequence. Then by the definition of the reading order, all crosses that appear at or below row $r$ in the $(a-1)$st diagonal of the corresponding pipe dream appear before $(a,r)$ in the compatible sequence. Therefore the inequality for the $(r,a-r+1)$ position is implied by the pseudo-Yamanouchi property. The converse is true by a similar argument.

By definition, $\wt(T)=\wt(D)$, where $T$ is a TSSCPP and $D$ its corresponding pseudo-Yamanouchi pipe dream. So this bijection is weight-preserving.
\end{proof}
See Figure~\ref{fig:topPDex} for an example.

\section{A bijection between TSSCPP and ASM in the reduced, 1432-avoiding case}
\label{sec:main}
In this section, we prove our first main result, Theorem~\ref{thm:main}. 
The proof uses the following theorem and lemmas, the first of which is
due to Yibo Gao. 

We need the following terminology.
A permutation $\pi$ \textbf{avoids} a permutation $\pi'$ if there is no subsequence of $\pi$ having the same relative order as $\pi'$. 

\begin{theorem}[\protect{\cite[Theorem 4.1]{gao1432}}]
\label{thm:simpleslides1432}
If $\pi\in S_n$ avoids $1432$,  then any two reduced pipe dreams of $\pi$ are connected by simple slides.
\end{theorem}

\begin{lemma}
Suppose $D\in\PD(n)$ is pseudo-Yamanouchi and $D'\in\PD(n)$ is related to $D$ by a simple slide. Then $D'$ is pseudo-Yamanouchi.  
\end{lemma}
\begin{proof}
Suppose $D\in\PD(n)$ is pseudo-Yamanouchi. Let $(\mathbf{a},\mathbf{r})=((a_1,\cdots, a_\ell), (r_1,\cdots, r_\ell))$ be its associated bounded compatible sequence. Suppose for some $1<i<n$, $D$ has a $\+$ tile at position $(r_i,a_i+1-r_i)$  and no $\+$ tiles at positions $(r_i,a_i+2-r_i)$, $(r_{i}-1,a_i-r_i)$, or $(r_{i}-1,a_i+1-r_i)$. Then a simple slide may be applied to $D$, resulting in another pipe dream $D'$ with $\+$ tile at position $(r_{i}-1,a_i+1-r_i)$  and no $\+$ tiles at positions $(r_{i},a_i+1-r_i)$, $(r_i,a_i+2-r_i)$, or $(r_{i}-1,a_i-r_i)$. That is, the simple slide moves the $\+$ tile up one unit and to the right one unit and there were no other $\+$ tiles in these intermediate squares. This preserves the diagonal but decrements the row index, 
creating 
a new bounded compatible sequence $(\mathbf{a}',\mathbf{r}')=((a_1',\cdots, a_{\ell}'), (r_1',\cdots, r_{\ell}'))$ 
such that
\begin{equation}
\label{eq:ak}
a_k'=\begin{cases}
a_k & k<i \\
a_{k+1} & i\leq k < m \\
a_i & k=m \\
a_k & m<k\leq\ell 
\end{cases}
\end{equation}
where $m>i$ is uniquely chosen so that $(\mathbf{a}',\mathbf{r}')$ satisfies the conditions of a bounded compatible sequence. Let $\tilde{j}:=a_i$. So $\mathbf{a}$ and $\mathbf{a}'$ differ only in that $\tilde{j}$ has slid to the right from index $i$ to $m$.

Recall $\cnt(k,j)(\mathbf{a})$ denotes the number of $j$ that appear in $a_1,\cdots, a_k$. 
By assumption, $(\mathbf{a}, \mathbf{r})$ is pseudo-Yamanouchi, so for all $1\le k\le \ell$, $1\le j\le n-2$, $1+\cnt(k,j)(\mathbf{a})\ge \cnt(k,j+1)(\mathbf{a})$. We need only check that $(\mathbf{a}', \mathbf{r}')$ is also pseudo-Yamanouchi. The only values of $j$ we need to consider are $\tilde{j}-1$, $\tilde{j}$, and $\tilde{j}+1$.

By (\ref{eq:ak}), $\cnt(k,j)(\mathbf{a}')=\cnt(k,j)(\mathbf{a})$ for all values of $j$ when $k<i$ or $k\geq m$. Thus we need only check the pseudo-Yamanouchi inequality for $k$ in the range $i\leq k<m$.

Suppose $i\leq k < m$. Since $(\mathbf{a},\mathbf{r})$ and $(\mathbf{a}',\mathbf{r}')$ are related by a simple slide, we know there is no cross in $D$ at position $(r_{i-1}, a_i - r_i)$. That is, $a_{m-1}\neq \tilde{j}-1$. Furthermore, $a_k\neq \tilde{j}-1$ for all $i\leq k < m$. Thus $\cnt(k,\tilde{j}-1)(\mathbf{a}')=\cnt(k,\tilde{j}-1)(\mathbf{a})$ for all $i\leq k < m$ while $\cnt(k,\tilde{j})(\mathbf{a}')=\cnt(k,\tilde{j})(\mathbf{a})-1$ in this same range. So \[\cnt(k,\tilde{j}-1)(\mathbf{a}')=\cnt(k,\tilde{j}-1)(\mathbf{a}) \ge \cnt(k,\tilde{j})(\mathbf{a})-1 =\cnt(k,\tilde{j})(\mathbf{a}').\] So the pseudo-Yamanouchi condition is more than satisfied when comparing diagonals $\tilde{j}-1$ and $\tilde{j}$.

Since $(\mathbf{a},\mathbf{r})$ and $(\mathbf{a}',\mathbf{r}')$ are related by a simple slide, we also know there is no cross in $D$ at position $(r_i, a_i+2 - r_i)$. That is, $a_{i+1}\neq \tilde{j}+1$. Furthermore, $a_k\neq \tilde{j}+1$ for all $i\leq k < m$. Thus $\cnt(k,\tilde{j}+1)(\mathbf{a}')=\cnt(k,\tilde{j}+1)(\mathbf{a})$ for all $i\leq k < m$ while $\cnt(k,\tilde{j})(\mathbf{a}')=\cnt(k,\tilde{j})(\mathbf{a})-1$ in this same range. So \[1+\cnt(k,\tilde{j})(\mathbf{a}')=1+\cnt(k,\tilde{j})(\mathbf{a}) \ge \cnt(k,\tilde{j}+1)(\mathbf{a}) =\cnt(k,\tilde{j}+1)(\mathbf{a}').\] Thus, the pseudo-Yamanouchi condition is satisfied on diagonals $\tilde{j}$ and $\tilde{j}+1$.

Therefore, $(\mathbf{a}',\mathbf{r}')$ is pseudo-Yamanouchi, implying $D'$ is pseudo-Yamanouchi.
\end{proof}

\begin{lemma}
\label{lem:av1432}
If $\pi\in S_n$ avoids  $1432$,  then all reduced pipe dreams of $\pi$ are pseudo-Yamanouchi.
\end{lemma}
\begin{proof}
Choose $\pi\in S_n$ that avoids  $1432$. Using the previous two lemmas, we know that simple slides preserve the pseudo-Yamanouchi property and that all reduced pipe dreams in $\PD^{red}(\pi)$ are connected by simple slides. So we need only show one reduced pipe dream is pseudo-Yamanouchi, and then all of them are. By Lemma~\ref{lem:bottom_Yam}, the bottom (permutation) pipe dream is pseudo-Yamanouchi. Thus the lemma is proved.
\end{proof}

%Recall that $ASM^{red}(\pi)$ denotes the set of ASM whose associated bumpless pipe dream is reduced and has permutation $\pi$.

%\begin{maintheorem}
%Let $\pi\in S_n$. There is an explicit weight-preserving injection $\varphi$ from $TSSCPP^{red}(\pi)$ to $ASM^{red}(\pi)$.
%If $\pi$ avoids $1432$, then $\varphi$ is a bijection.
%\end{maintheorem}
\begin{proof}[Proof of Theorem~\ref{thm:main}] Let $\pi\in S_n$.
The explicit bijection $\varphi:\PD^{red}(\pi)\rightarrow\BPD^{red}(\pi)$ of \cite{GH} discussed in Section~\ref{sec:BPDPD_bij} is weight-preserving; in particular, for $D\in\PD^{red}(\pi)$, the number of cross-tiles in row $k$ equals the number of blank tiles in row $k$ of $\varphi(D)$, so $\wt(D)=\wt(\varphi(D))$. By Theorem~\ref{thm:TSSCPP_Yam}, TSSCPP are characterized as the set of pseudo-Yamanouchi pipe dreams in $\PD(n)$, and the weight is preserved in this bijection. Thus whenever such pipe dreams are reduced, $\varphi$ produces a BPD with the same weight, which is in bijection with an ASM of the same weight. Thus we have a weight-preserving injection $\varphi:\TSSCPP^{red}(\pi)\hookrightarrow  \ASM^{red}(\pi)$ given by transfoming the TSSCPP to its corresponding reduced pipe dream as in Theorem~\ref{thm:TSSCPP_Yam}, mapping it to a reduced BPD using $\varphi$, and then transforming to an ASM using the bijection described in Section~\ref{sec:bpd}.

Suppose $\pi$ avoids $1432$. Then by Lemma~\ref{lem:av1432}, all pipe dreams in $\PD^{red}(\pi)$ are pseudo-Yamanouchi, so $\TSSCPP^{red}(\pi)$ is in bijection with $\PD^{red}(\pi)$. So the above injection is a bijection between $\TSSCPP^{red}(\pi)$ and $\ASM^{red}(\pi)$.
\end{proof}

\section{A poset-preserving bijection in the $2143$- and $1432$-avoiding case}
\label{sec:poset}

In this section, we study posets constructed using simple slides on pipe dreams and {simple droops} on bumpless pipe dreams. We use this understanding to construct a bijection between $\TSSCPP^{red}(\pi)$ and $\ASM(\pi)$ that preserves their poset structure in the case that $\pi$ avoids both $1432$ and $2143$. Note that this bijection is not, in general, the same as the bijection $\varphi$ used in Theorem~\ref{thm:main}; see Remark~\ref{remark:varphi}. Note also that when $\pi$ avoids $2143$, $\ASM^{red}(\pi)=\ASM(\pi)$, since all BPD that avoid $2143$ are reduced; see Lemma~\ref{lem:lemma74} below.

We use simple droops to define a poset on BPD. The \textbf{Rothe BPD} of $\pi$ is the BPD corresponding to the permutation matrix of $\pi$. We will also refer to the Rothe BPD as the \textbf{Rothe diagram} of $\pi$. Note that the connected regions of blank tiles of a Rothe diagram are all partition-shaped.

\begin{definition}
Given $\pi\in S_n$, let $\droop(\pi)$ denote the poset constructed from applying simple droops in all possible ways to the Rothe BPD of $\pi$. The Rothe BPD is the bottom element of the poset and each simple droop moves up in the poset. 
\end{definition}
Note that this construction also induces a poset on the corresponding ASM, by the simple bijection between BPD and ASM.

\begin{definition}
The set of \textbf{essential boxes} of a permutation $\pi$ is the set $\ess(\pi):=\{(i,j):\pi(j)>i, \pi^{-1}(i)>j, \pi(j+1)\le i, \pi^{-1}(i+1)\le j\}$. In other words, $\ess(\pi)$ consists of the SE-most corners in the connected regions of blank tiles of the Rothe diagram of $\pi$. Define the \textbf{dominant region} of Rothe diagram to be the connected region of blank tiles containing $(1,1)$. Note that the dominant region might be empty.
\end{definition}

The following statement is found in \cite{weigandt}. 
\begin{lemma}[\protect{\cite[Lemma 7.2, Lemma 7.4 (2)]{weigandt}}]
\label{lem:lemma74}
If $\pi\in S_n$ avoids $2143$,  then any $D\in\BPD(\pi)$ is reduced, and any
two  bumpless pipe dreams of $\pi$ are connected by simple droops. 
\end{lemma}

This yields the following corollary.
\begin{corollary}
\label{cor:droop_is_BPD}
If $\pi\in S_n$ avoids $2143$, the elements of $\droop(\pi)$ are all of $\BPD(\pi)$.
\end{corollary}

We define a similar poset for pipe dreams, using the simple slides of Section~\ref{sec:pd}.
\begin{definition}
Given $\pi\in S_n$, let $\slide(\pi)$ denote the poset constructed from applying simple slides to the bottom  pipe dream in all possible ways. The bottom pipe dream is the bottom element of this poset, and each simple slide moves up in the poset. 
\end{definition}

We have the following, as a corollary of the result of Gao we stated as Theorem~\ref{thm:simpleslides1432}.
\begin{corollary}
If $\pi\in S_n$ avoids $1432$, the elements of $\slide(\pi)$ are all of $\PD^{red}(\pi)$.
\end{corollary}

An important class of permutations are the \textbf{Grassmannian} permutations, which are defined as the permutations with at most one descent. Grassmannian permutations necessarily avoid both $2143$ and $1432$, since these patterns each have two descents. \textbf{Inverse-Grassmannian} permutations are permutations whose inverse is Grassmannian. These also avoid both $2143$ and $1432$, since these patterns are their own inverses.

We now relate $\droop(\pi)$ and $\slide(\pi)$ in the cases that $\pi$ is inverse-Grassmannian or Grassmannian. We will need the following lemmas, the first of which is also due to Weigandt. 
\begin{lemma}[\protect{\cite[Lemma 7.4 (1)]{weigandt}}]
\label{lem:blankdetermines}
If $\pi$ avoids $2143$, all bumpless pipe dreams in $\BPD(\pi)$ are uniquely determined by the locations of the blank tiles.
\end{lemma}

We remark that the same paper shows $1432$-avoiding bumpless pipe dreams are in bijection with flagged tableaux constructed by filling the blank tiles with numbers~\cite[Theorem 1.6]{weigandt}.

\begin{lemma}
\label{lem:GrassRotheblank}
The blank tiles in the Rothe BPD of a Grassmannian permutation have the following characterizing properties:
\begin{enumerate}
\item[(a)] Each connected region of blank tiles is a rectangular block.
\item[(b)] The essential boxes lie in the same row.
\item[(c)] The NW-most blank tile of the leftmost rectangular block is on the diagonal. 
\item[(d)] If $B_1$ and $B_2$ are two consecutive blocks and the horizontal distance between them is $d$, then the height of $B_1$ is $d$ more than the width of $B_2$.
\end{enumerate}
Furthermore, if a set of blank tiles satisfy the properties above, this set uniquely determines the Rothe BPD of a Grassmannian permutation. If we replace (c) and (d) with the following
\begin{enumerate}
\item[(b$'$)]  The essential boxes lie in the same column.
\item[(d$'$)] If $B_1$ and $B_2$ are two consecutive blocks and the vertical distance between them is $d$, then the width of $B_1$ is $d$ more than the width of $B_2$.
\end{enumerate}
we get similar characterizing properties for the inverse-Grassmannian permutations.
\end{lemma}
\begin{proof}

We may consider the permutation as an element in $S_\infty$ and the BPDs to be unbounded in the east and south directions.
The Rothe BPD of an inverse Grassmannian permutation is the transpose of that of a Grassmanian permutation, so it suffices to argue for Grassmannian permutations. It is easy to see the set of blank tiles of the Rothe BPD for a Grassmannian permutation satisfies properties (a)--(d). Given a set of blank tiles satisfying (a)--(d), we may construct a Rothe BPD by iteratively placing a $\rt$ tile at the NW-most undetermined position, and extending to the east by a horizontal ray and to the south by a vertical ray. The properties (a)--(d) guarantee that this is always possible, see Figure~\ref{fig:blocks-determine-grass}.
\end{proof}

\begin{figure}[htbp]
\includegraphics[scale=0.8]{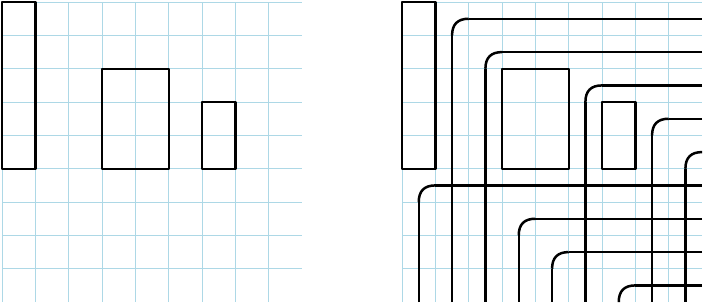}
\caption{The Rothe BPD of a Grassmannian permutation is determined by blank tiles satisfying conditions (a)--(d) of Lemma~\ref{lem:GrassRotheblank}. Transposing the diagrams we get the analogous statements for inverse Grassmannians.}
\label{fig:blocks-determine-grass}
\end{figure}

We are now ready to show a poset isomorphism between BPD and reduced PD in the inverse-Grassmannian case.
\begin{theorem}
\label{thm:invGrass}
If $\pi\in S_n$ is inverse-Grassmannian, there is an explicit weight-preserving bijection between $\BPD(\pi)$ and $\PD^{red}(\pi)$ such that 
 $\droop(\pi)$ $\cong$ $\slide(\pi)$. 
\end{theorem}
\begin{proof}
Suppose $\pi\in S_n$ is inverse-Grassmannian. By Lemma~\ref{lem:blankdetermines}, since $\pi$ avoids $2143$, the blank tiles of the BPD completely determine it.  The map on the Rothe $\BPD$ of $\pi$ that left-justifies all the blank tiles and turns them into crosses results in the bottom (permutation) pipe dream of $\pi$~\cite{BB}. By Lemma~\ref{lem:GrassRotheblank}, the blank tiles in the Rothe BPD of $\pi$ are disconnected rectangular blocks that are aligned on the right. When they are left-justified to create the bottom pipe dream, they become mirror-image rectangular blocks aligned on the left. Furthermore, the NW-most blank tile in the Rothe $\BPD$ is on the diagonal. The number of simple droops that may be applied to this blank tile then equals the number of simple slides that may be applied to the rightmost cross in this row of the bottom pipe dream. Moreover, the simple droop moves of $\droop(\pi)$ 
correspond exactly to the simple slides of $\slide(\pi)$. 
Thus, $\droop(\pi)$ $\cong$ $\slide(\pi)$. By construction, the number of blank tiles in row $k$ of a BPD in $\droop(\pi)$ corresponds to the number of cross tiles in row $k$ of the corresponding PD, so this weight is preserved. See Figure~\ref{fig:invGrass-BPD-PD} for an example.
\end{proof}

The following is a direct application of Theorem~\ref{thm:invGrass} to ASM and TSSCPP.
\begin{corollary}
\label{cor:invGrass}
If $\pi\in S_n$ is inverse-Grassmannian, 
there is an explicit weight-preserving bijection between $\ASM(\pi)$ and $\TSSCPP^{red}(\pi)$ such that   $\droop(\pi)$ $\cong$ $\slide(\pi)$.
\end{corollary}

\begin{figure}[htbp]
\includegraphics[scale=0.57]{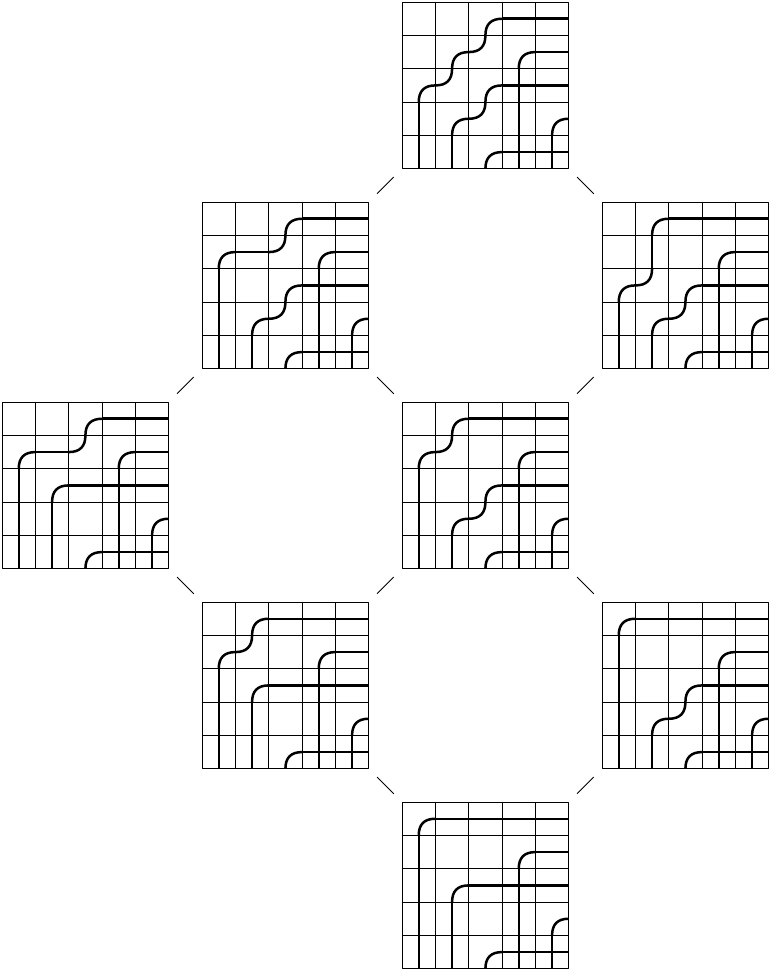}\quad
\includegraphics[scale=0.57]{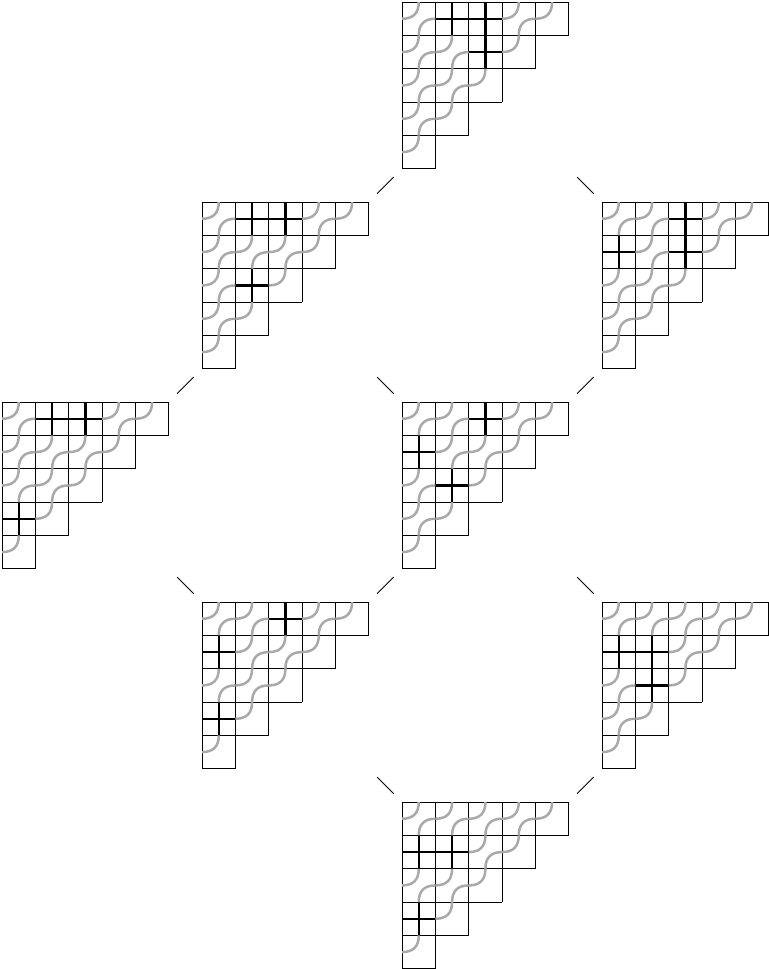}
\caption{A weight-preserving and poset-preserving bijection between $\BPD(14253)$ and $\PD(14253)$. Note that $14253$ is inverse-Grassmannian and avoids both $1432$ and $2143$.}
\label{fig:invGrass-BPD-PD}
\end{figure}

\begin{remark}
\label{remark:varphi}
For crystal-theoretic reasons, the bijection of Corollary~\ref{cor:invGrass} coincides with the bijection $\varphi$ of Theorem~\ref{thm:main}, though this will not be true for the subsequent corollaries of this section.
\end{remark}

We have a similar result for Grassmannian permutations, but the poset isomorphism is with the \textbf{dual poset} $\slide(\pi)^*$ (all order relations reversed). Note that in the case that $\pi$ avoids 1432, by Theorem~\ref{thm:simpleslides1432}, $\slide(\pi)$ has a unique maximal element, the pipe dream in which all the crosses are top-justified; this is called the \textbf{top pipe dream}. Thus, $\slide(\pi)^*$ has the top pipe dream as its unique minimal element.

We now show a dual poset isomorphism between BPD and reduced PD in the Grassmannian case. Note that the transpose of a Grassmannian permutation is inverse-Grassmannian, so Grassmannian permutations satisfy the transpose of the description in Theorem~\ref{thm:invGrass}.
\begin{theorem}
\label{thm:Grass}
If $\pi\in S_n$ is Grassmannian, there is an explicit 
bijection between $\BPD(\pi)$ and $\PD^{red}(\pi)$ such that 
 $\droop(\pi)$ $\cong$ $\slide(\pi)^*$. 
\end{theorem}
\begin{proof}
Suppose $\pi\in S_n$ is Grassmannian. By Lemma~\ref{lem:blankdetermines}, since $\pi$ avoids $2143$, the blank tiles of the BPD completely determine it.  The map on the Rothe $\BPD$ of $\pi$ that top-justifies all the blank tiles and turns them into crosses results in the top pipe dream of $\pi$. This can be seen by transposing the construction that left-justifies  blank tiles of inverse-Grassmannian Rothe BPDs as described in the proof of Theorem~\ref{thm:invGrass}. By Lemma~\ref{lem:GrassRotheblank}, the blank tiles in the Rothe BPD of $\pi$ are disconnected rectangular blocks that are aligned on the bottom. When they are top-justified to create the top pipe dream, they become mirror-image rectangular blocks aligned on the top. Furthermore, the NW-most blank tile in the Rothe $\BPD$ is on the diagonal. The number of simple droops that may be applied to this blank tile then equals the number of inverse simple slides that may be applied to the lowest cross in this column of the top pipe dream. Moreover, the simple droop moves of $\droop(\pi)$ 
correspond exactly to the inverse simple slides of $\slide(\pi)$. 
Thus, $\droop(\pi)$ $\cong$ $\slide(\pi)^*$. 
\end{proof}

\begin{remark}
Note that by construction of this poset isomorphism, the number of blank tiles in row $k$ of a BPD in $\droop(\pi)$ corresponds to the number of cross tiles row $n-k+1$ of the corresponding PD, so this weight is reversed instead of preserved.
\end{remark}

The corollary below follows immediately from Theorem~\ref{thm:Grass}.
\begin{corollary}
\label{cor:Grass}
If $\pi\in S_n$ is Grassmannian, 
there is an explicit bijection between $\ASM(\pi)$ and $\TSSCPP^{red}(\pi)$ such that   $\droop(\pi)$ $\cong$ $\slide(\pi)^*$.
\end{corollary}

\begin{remark}
One may wonder why we use the dual poset in this theorem/corollary when many small examples, including that of Figure~\ref{fig:invGrass-BPD-PD}, seem to indicate that these posets may be self-dual. An example of a Grassmannian permutation $\pi$ for which $\slide(\pi)\neq \slide(\pi)^*$ is $\pi=146235$.
\end{remark}

Now we generalize these correspondences to define a poset on $\PD(\pi)$ in the case that $\pi$ avoids both $1432$ and $2143$ that is isomorphic to $\droop(\pi)$. This will yield a poset-preserving bijection between the corresponding TSSCPP and ASM. First, we characterize features of these BPD in the lemmas below.

The following statement is well-known. We include a short proof.
\begin{lemma}
\label{lem:2143-NE-SW}
If $\pi$ is $2143$-avoiding, then essential boxes of the Rothe diagram of $\pi$ run northeast to southwest. Namely, it is impossible that $(i_1,j_1), (i_2,j_2)\in \ess(\pi)$ such that $i_1<i_2$ and $j_1<j_2$.
\end{lemma}
\begin{proof}
Let $\pi$ be a permutation and suppose $(i_1,j_1)$ and $(i_2,j_2)$ are two essential boxes of $\pi$. Suppose $i_1<i_2$ and $j_1<j_2$. Since $(i_1,j_1)\in\ess(\pi)$, we have $\pi^{-1}(j_1+1)\le i_1$ and $\pi(i_1+1)\le j_1$. Since $(i_2,j_2)\in\ess(\pi)$, there exist $i_1+1<a\le i_2$ such that $\pi(a)>j_2$ and $j_1+1<b\le j_2$ such that $\pi^{-1}(b)>i_2$. We then have $\pi^{-1}(j_1+1)<i_1+1<a<\pi^{-1}(b)$ where $\pi(i_1+1)<j_1+1<b<\pi(a)$, so $\pi$ must be 2143-containing.
\end{proof}

\begin{lemma}
\label{lem:1432-avoid}
If $\pi$ is $1432$-avoiding, the Rothe diagram of $\pi$ satisfies the following properties:
\begin{enumerate}
\item[(a)] The partition shapes formed by the blank region must all be rectangles, except for the dominant region. 
\item[(b)] Let $B_1$ and $B_2$ be two non-dominant connected regions of blank tiles with essential boxes $(i_1,j_1)$ and $(i_2,j_2)$.  If $B_1$ and $B_2$ both contain tiles in some row $r$, then $i_1=i_2$. If $B_1$ and $B_2$ both contain tiles in some column $j$, then $j_1=j_2.$
\end{enumerate}
\end{lemma}

\begin{proof}
For (a), suppose to the contrary that there is a non-rectangular partition that is not dominant. Let $(a,b)$ be its NW corner, and $(i_1,j_1)$, $(i_2,j_2)$ with $i_1<i_2$ and $j_1>j_2$ SE corners of this partition with no other SE corners in between. Note that $a>1$. Let $(x,y)$ be the unique outer NW corner SW of $(i_1,j_1)$ and NE of $(i_2,j_2)$. It must be the case that $\pi(x)=y$. Then $\pi(a-1)<b$, $\pi(i_1)>j_1$, and $\pi^{-1}(j_2)>i_2$. Therefore, $a-1<i_1<x<\pi^{-1}(j_2)$ and $\pi(a-1), \pi(i_1), \pi(x), j_2$ gives the pattern 1432, contradicting our assumption on $\pi$.

For (b), we show the row statement and the column statement is similar. It suffices to consider the case when $B_1$ and $B_2$ are adjacent and $j_1<j_2$. Now suppose $i_1\neq i_2$. Since $B_1$ and $B_2$ both contain tiles in some row $r$, we must have $i_1>i_2$. Let $(a,b)$ be the NW corner of $B_1$, then $\pi(a-1)<b\le j_1$. Since $(i_2,j_2)\in\ess(\pi)$, $\pi(i_2)>j_2$. Note also that $j_1<\pi(i_1)\le j_2$. Finally, $\pi^{-1}(b)>i_1$. We then have $a-1<i_2<i_1<\pi^{-1}(b)$ and $\pi(a-1)<b<\pi(i_1)<\pi(i_2)$ which gives the pattern 1432.
\end{proof}

\begin{lemma}
\label{lem:2143-1432}
Suppose $\pi$ is $1432$- and $2143$-avoiding.
If $(i,j_1),(i,j_2)\in \ess(\pi)$, $j_1<j_2$ and $(i,j_1)$ is not in the dominant region, then $(i',j_1)\not\in\ess(\pi)$ for any $i'>i$. Similarly if $(i_1,j), (i_2,j)\in \ess(\pi)$, $i_1<i_2$ and $(i_1,j)$ is not in the dominant region, then $(i_1,j')\not\in\ess(\pi)$ for any $j'>j$.

\end{lemma}
\begin{proof}
Suppose $(i_1,j_1),(i_1,j_2),(i_2,j_1)\in\ess(\pi)$ with $i_1<i_2$ and $j_1<j_2$, and $(i_1,j_1)$ is not in the dominant region. Suppose $B$, $B_1$, $B_2$ are the blank regions containing $(i_1,j_1)$, $(i_1,j_2)$, and $(i_2,j_1)$, respectively. Suppose $(a_1,b)$ is the NW corner of $B$ and $a_2$ is the row index of the topmost row in $B_2$. Then $\pi(a-1)<b\le j_1$, $\pi(i_1)>j_2$, and $\pi^{-1}(j_1)>i_2$. Since the essential boxes of 2143-avoiding permutations run NE to SW by Lemma~\ref{lem:2143-NE-SW}, there can be no blank tiles strictly SE of $(i_1,j_1)$, and therefore $j_1<\pi(a_2)\le j_2$. We then have $a-1<i_1<a_2<\pi^{-1}(j_1)$ and $\pi(a-1)<j_1<\pi(a_2)<\pi(i_2)$ giving the pattern 1432.
\end{proof}

\begin{deflemma}
\label{lem:block_decomp}
If $\pi\in S_n$ is $2143$- and $1432$-avoiding, the Rothe $\BPD$ can be partitioned into a dominant partition of empty squares in the NW corner, a partition shape of fixed pipes in the SE corner, and disjoint blocks containing non-intersecting pipes corresponding to Grassmannian and inverse-Grassmannian permutations. Call this the \textbf{block-decomposition} of $\pi$.
\end{deflemma}
\begin{proof}
By Lemmas~\ref{lem:1432-avoid} and \ref{lem:2143-1432}, the set of essential boxes outside of the dominant region can be partitioned into $E_1,\cdots, E_m$ according to whether they lie on the same row or column. In particular, if $e_1,e_2\in E_i$ for some $i$, then either $e_1$ and $e_2$ are in the same row, or they are in the same column.  If $E_i$ contain essential boxes in the same row $r$, let $c$ be the smallest column index such that $(r,c)$ is not in the dominant region. Let $r'<r$ be the smallest row index such that $(r',c)$ is not in the dominant region. Let $l$ be the sum of the widths of the rectangular blocks of blank tiles with essential boxes in $E_i$. Then the rectangular region with $(r,c)$ as the SW corner of height $r-r'+1$ and width $l+r-r'+1$ consist of $r-r'+1$ non-intersecting pipes and agree with the top $r-r'+1$ rows of the Rothe BPD of a Grassmannian permutation by Lemma~\ref{lem:GrassRotheblank}. Transposing this construction we get a rectangular region of inverse-Grassmannian permutation. The rest of the Rothe diagram contains fixed pipes, as there are no more blank tiles.
\end{proof}

\begin{theorem}
\label{thm:poset_bij}
Suppose $\pi\in S_n$ avoids both $2143$ and $1432$. Let $\sigma_1,\sigma_2,\ldots,\sigma_k$ be the inverse-Grassmannian permutations  in the block-decomposition of $\pi$ and $\tau_1,\tau_2,\ldots,\tau_{\ell}$ the Grassmannian permutations in the block-decomposition of $\pi$. There is an explicit bijection between $\BPD(\pi)$ and $\PD^{red}(\pi)$ 
such that 
\[\droop(\pi) \cong \slide(\sigma_1)\times \slide(\sigma_2)\times \cdots\times \slide(\sigma_k)\times \slide(\tau_1)^*\times \slide(\tau_2)^*\times \cdots\times \slide(\tau_{\ell})^*.\]
\end{theorem}
\begin{proof} 
By Lemma~\ref{lem:block_decomp}, the Rothe diagram of such a permutation is block-Grassmannian and inverse-Grassmannian. So we begin by mapping the Rothe diagram to a PD block-by-block. For the Grassmannian blocks, we map to the top pipe dream, and for the inverse-Grassmannian blocks, we map to the bottom pipe dream. Then from Theorems~\ref{thm:invGrass} and \ref{thm:Grass}, the posets for each block are isomorphic or dual isomorphic. So we have the stated poset isomorphism. 
\end{proof}

The corollary below follows immediately. 
\begin{corollary}
\label{cor:ASMTSSCPPposet_bij}
If $\pi\in S_n$ avoids both $2143$ and $1432$, 
there is an explicit bijection $\psi$ between  $\ASM(\pi)$ and $\TSSCPP^{red}(\pi)$ such that
\[\droop(\pi) \cong \slide(\sigma_1)\times \slide(\sigma_2)\times \cdots\times \slide(\sigma_k)\times \slide(\tau_1)^*\times \slide(\tau_2)^*\times \cdots\times \slide(\tau_{\ell})^*\]
where $\sigma_1,\sigma_2,\ldots,\sigma_k$ are the inverse-Grassmannian permutations and $\tau_1,\tau_2,\ldots,\tau_{\ell}$ the Grassmannian permutations in the block-decomposition of $\pi$.
\end{corollary}

\begin{example}
In Figure~\ref{fig:avoid-both-ex}, we show an example of  the Rothe BPD for a 2143- and 1432- avoiding permutation $\pi$ and its image under the poset-preserving bijection $\psi$. The dominant partition in the upper left corner of each diagram is $(6,6,6,5,4)$. The block-decomposition of $\pi$ (from Definition-Lemma~\ref{lem:block_decomp}) consists of the Grassmannian permutation $\sigma=146235$ (shown in the upper right of each diagram) and the inverse-Grassmannian permutation $\tau=142563$ (shown in the lower left of each diagram). 
\end{example}

\begin{figure}
\includegraphics[scale=0.8]{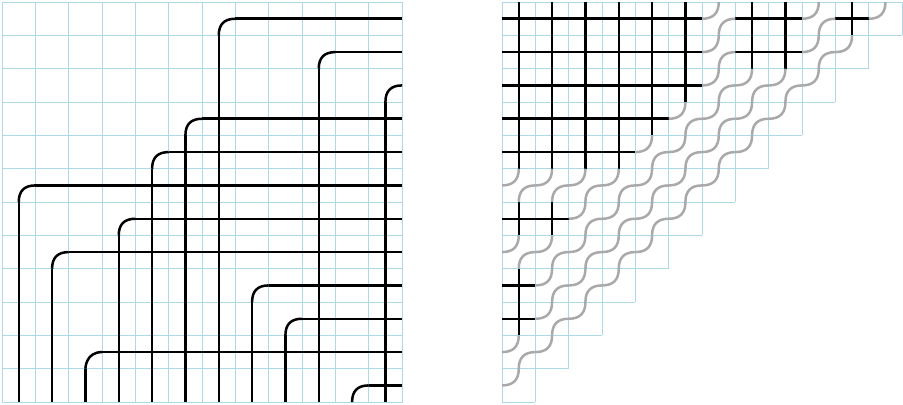}
\caption{Example of an Rothe BPD for a 2143- and 1432- avoiding permutation and its image under the poset-preserving bijection $\psi$}
\label{fig:avoid-both-ex}
\end{figure}

\begin{remark}
\label{remark:tab}
In  \cite{fan2021set}, Fan and Guo give a formula with set-valued Rothe tableaux for Grothendieck polynomials indexed by 1432-avoiding permutations, which when restricted to the reduced case gives a formula for Schubert polynomials. It is not hard to see that their (set-valued) Rothe tableaux are in direct bijection with pipe dreams. In a similar spirit, Weigandt \cite{weigandt} gives a direct bijection between vexillary bumpless pipe dreams and flagged (set-valued) tableaux. It is possible to combine these tools in the $2143$- and $1432$- avoiding case to obtain a bijection between pipe dreams and bumpless pipe dreams via tableaux. Such a  bijection agrees with the Gao-Huang bijection \cite{GH} and therefore is different from our poset-preserving bijection.
\end{remark}

\section{Concluding remarks}
\label{sec:remarks} 
We conclude the paper by discussing the progress made thus far on the ASM-TSSCPP bijection problem and an outlook on and challenges to further progress.
Table~\ref{tab1} below shows the numbers up to $n=7$ (computed using \verb|SageMath|~\cite{sage}) of ASM and TSSCPP that correspond using the theorems of this paper (Columns 3, 6, 7) as well as the results of \cite{PermTSSCPP,Ayyer312,StrikerPoset,Biane_Cheballah_1,Bettinelli} (Columns 2, 4, 5).
\begin{itemize}
\item Column 2 is given by $n!$, as this counts the permutation case bijection of \cite{PermTSSCPP}.
\item Column 3 gives those matched under the poset-preserving $1432$- and $2143$-avoiding bijection of the present paper (Corollary~\ref{cor:ASMTSSCPPposet_bij}).  
\item
Column 4 gives the number of ASM and TSSCPP matched in the bijection of~\cite[Theorem 4]{Ayyer312}, which concerns the case that the monotone triangle associated to the ASM is \emph{gapless}; this is the same set as the intersection discussed in \cite[Section 7]{StrikerPoset}. Remark~\ref{remark:213} below explains the connection to pattern avoidance.
\item Column 5 gives the number matched in the bijections of~\cite{Biane_Cheballah_1,Bettinelli}, which both concern the case of at most two nontrivial diagonals in the monotone triangle. 
\item Column 6 gives the number  matched in the weight-preserving $1432$-avoiding bijection of the present paper (Theorem~\ref{thm:main}). 
\item Column 7 gives the number matched in the weight-preserving injection of Theorem~\ref{thm:main} on all TSSCPP whose pipe dreams are reduced. 
\item Column 8 gives the total number of ASM, for comparison. 
\end{itemize}

\begin{table}[hbtp]
\begin{center}
\begin{tabular}{ |c|c|c|c|c|c|c|c| } 
 \hline
 Size & Perm & \textbf{(1432,2143)} &  $213$-avoiding & Two & \textbf{1432} & \textbf{Matched} &  Total \\ 
 & bijection & \textbf{-avoiding} & (gapless) & diagonal & \textbf{-avoiding} & \textbf{in} & number \\
 &  & \textbf{bijection} & bijection & bijection & \textbf{bijection} & \textbf{injection} &  of ASM \\
 \hline 
 1 & 1 & 1 & 1 & 1 & 1 & 1 & 1 \\ 
 2 & 2 & 2 & 2 & 2 & 2 & 2 & 2 \\ 
 3 & 6 & 7 & 6 & 7 & 7 & 7 & 7 \\ 
 4 & 24 & 33 & 26 & 35 & 36 & 40 & 42 \\ 
 5 & 120 & 185 & 162 & 219 & 246 & 362 & 429 \\ 
 6 & 720 & 1175 & 1450 &  1594 & 2135 & 5125 & 7436 \\ 
 7 & 5040 & 8261 & 18626  & 12935 & 23067 & 112941 & 218348 \\ 
 \hline
\end{tabular}
\end{center}
\caption{The number of ASM and TSSCPP in correspondence via the various results of this paper and as compared to other subset bijections. The column headings in bold represent results from this paper.}
\label{tab1}
\end{table}

One may ask whether any of these bijections include any of the other partial bijections. As noted earlier, the weight-preserving injection of Theorem~\ref{thm:main} (Column $7$) extends the permutation case bijection of~\cite{PermTSSCPP} (Column $2$). Also, the subsets addressed in this paper are proper subsets of each other (Columns 3, 6, 7). But it is useful to note that the subset of ASM included in other bijections discussed here may not be a proper subset of the ASM included in Theorem~\ref{thm:main}. In particular, the one ASM with $n=4$ whose BPD is non-reduced (pictured in Figure~\ref{fig:pathology}, upper right) has monotone triangle with only two non-trivial diagonals. Thus, it is included in the bijection of Column 5 \cite{Biane_Cheballah_1,Bettinelli}, but not in Theorem~\ref{thm:main} (Column~7).

\begin{remark}
\label{remark:213}
Pattern avoidance is discussed in \cite{Ayyer312}, in the sense that the bijection of Column $4$ \cite[Theorem 4]{Ayyer312} includes all permutations that avoid the pattern $312$, using the conventions of that paper. In the conventions of the current paper, this corresponds to the set $\{\ASM(\pi) \ | \ \pi \mbox{ avoids } 213\}$, the ASM whose associated permutation $\pi$ avoids the pattern $213$. (A gapless monotone triangle is obtained from such an ASM $A$ as follows: the $i$th row consists of the column indices whose partial sum \textit{from the bottom row}  to row $(n-i+1)$ of $A$ equals $1$.) 
%Columns 3 and 4 in Table~\ref{tab1}. 
If a permutation avoids $213$, it necessarily avoids $2143$, but it might not avoid $1432$. Table~\ref{tab1} shows that for small $n$, the cardinality of $\{\ASM(\pi) \ | \ \pi \mbox{ avoids } 213\}$ is smaller than the cardinality of $\{\ASM(\pi) \ | \ \pi \mbox{ avoids } 1432 \mbox{ and } 2143\}$, but for larger values of $n$ in the table, this comparison is reversed.
\end{remark}

One may ask whether it is possible to extend the bijection of Theorem~\ref{thm:main} beyond $1432$-avoiding permutations and/or remove the reducedness restrictions. There are some challenges. In the case of $n=4$, $40$ pseudo-Yamanouchi pipe dreams are reduced, so  all the corresponding TSSCPP are mapped to reduced BPD, and therefore to ASM. There are only two remaining TSSCPPs, shown in the left column of Figure~\ref{fig:pathology}. There is one remaining reduced pipe dream, shown in the bottom middle of Figure~\ref{fig:pathology}, which maps by $\varphi$ to the BPD on the bottom right. Comparing this reduced pipe dream with the two remaining TSSCPP pipe dreams, we see it differs from the one with three crosses by moving the top cross to the right, creating a non-reduced pipe dream. In general, the set of reduced pipe dreams for a fixed permutation are connected by $(n\times 2)$--ladder moves and $(2\times n)$--chute moves \cite{BB}, and these moves do not always preserve the pseudo-Yamanouchi property when $n>2$. 

In the forthcoming work of Shimozono, Yu, and the first author, a weight-preserving bijection between the set of all $2^{\binom{n}{2}}$ pipe dreams in $\PD(n)$ and the set of all marked bumpless pipe dreams (a marked BPD is a BPD whose $\jt$-tiles admit a binary marking) of size $n$, which generalizes the bijection in \cite{GH}. Under this bijection, the BPD on the top right of Figure~\ref{fig:pathology} is mapped to the PD on the top middle. Notice that the one remaining TSSCPP PD on the top left has one more $\+$, so a bijection that preserves this weight is no longer possible.
\begin{figure}
\includegraphics[scale=0.8]{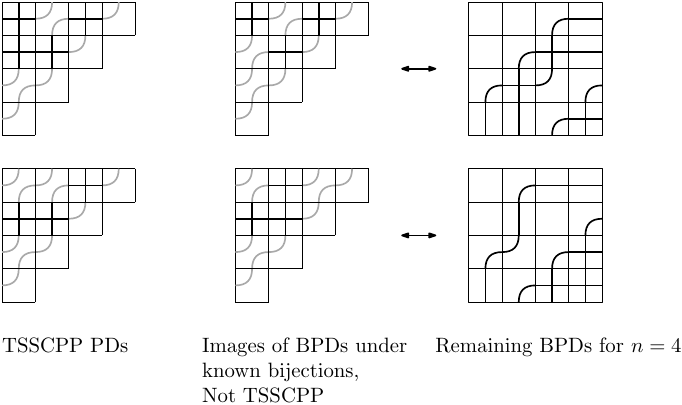}
\caption{Challenges for the remaining unmatched TSSCPP pipe dreams and bumpless pipe dreams}
\label{fig:pathology}
\end{figure}

\section*{Acknowledgments}
The authors would like to thank the anonymous referees for helpful comments. They also thank Yibo Gao for helpful comments, Mathilde Bouvel and Rebecca Smith for comments on pattern avoidance and gapless monotone triangles (which resulted in Remark~\ref{remark:213}), and Anna Weigandt for comments on connections to tableaux (which precipitated Remark~\ref{remark:tab}). 
They also thank the developers of \verb|SageMath|~\cite{sage} software, which was useful in this research. They thank O.~Cheong for developing Ipe~\cite{ipe}, which was used to create the figures in this paper. JS was supported by a grant from the Simons Foundation/SFARI (527204, JS) and NSF grant DMS-2247089. DH was supported by NSF grant DMS-2202900.

\bibliographystyle{alpha}
\bibliography{main}

\end{document}